\documentclass{gtpart}     
\agtart
\usepackage{pinlabel}  

\usepackage{amsmath}
\usepackage{amssymb,mathtools,amsthm}
\usepackage{comment,xcolor}
\usepackage[all]{xy}
\usepackage{graphicx}
\usepackage{amssymb}
\usepackage{mathrsfs}
\usepackage{multirow}
\usepackage{float}
\usepackage{tikz-cd}
\usepackage{adjustbox}
\usepackage{amsthm}

\title{On ${\rm mod}~p$ $A_p$-spaces}

\author{Ruizhi Huang}
\givenname{Ruizhi}
\surname{Huang}
\address{Department of Mathematics\\ National University of Singapore\\ 10 Lower Kent Ridge Road\\ Singapore 119076}
\email{a0123769@u.nus.edu}
\urladdr{https://sites.google.com/site/hrzsea/}

\author{Jie Wu}
\givenname{Jie}
\surname{Wu}
\address{Department of Mathematics\\ National University of Singapore\\ 10 Lower Kent Ridge Road\\ Singapore 119076}
\email{matwuj@nus.edu.sg}
\urladdr{http://www.math.nus.edu.sg/~matwujie}


\arxivreference{}
\arxivpassword{}

\volumenumber{}
\issuenumber{}
\publicationyear{}
\papernumber{}
\startpage{}
\endpage{}
\doi{}
\MR{}
\Zbl{}
\received{}
\revised{}
\accepted{}
\published{}
\publishedonline{}
\proposed{}
\seconded{}
\corresponding{}
\editor{}
\version{}

\usepackage[numbers,sort&compress]{natbib} 

\newcounter{RomanNumber}
\newcommand{\MyRoman}[1]{\setcounter{RomanNumber}{#1}\Roman{RomanNumber}}

\newtheorem{theorem}{Theorem}[section]
\newtheorem{lemma}[theorem]{Lemma}

\theoremstyle{definition}
\newtheorem{definition}[theorem]{Definition}

\newtheorem{proposition}[theorem]{Proposition}
\newtheorem{corollary}[theorem]{Corollary}

\theoremstyle{remark}
\newtheorem{remark}[theorem]{Remark}

\theoremstyle{notation}

\numberwithin{equation}{section}

\begin{document}

\begin{abstract}    
We prove a necessary condition for the existence of the $A_p$-structure on ${\rm mod}~p$ spaces, and also derive a simple proof for the finiteness of the number of ${\rm mod}~p$ $A_p$-spaces of given rank. As a direct application, we compute a list of possible types of rank $3$ ${\rm mod}~3$ homotopy associative $H$-spaces.  
\end{abstract}

\maketitle

\section{Introduction}
\noindent A longstanding problem in algebraic topology is to classify finite $H$-spaces. However, this problem is rather complicated, and only been solved in few cases. There is Zabrodsky's localization and mixing theorem \cite{Zabro} yielding that a simply connected finite complex is an $H$-space if and only if each of its $p$-localizations is an $H$-space. One would also like to know for which primes $p$ the localization at $p$ fails to be an $H$-space, so it is natural to consider the $p$-local version of $H$-spaces.

Let $X$ be a $CW$-complex with cohomology being an exterior algebra generated by $r$ elements of odd dimension, we call $r$ the rank of $X$. For $r=1$,  J. F. Adams has determined that $S^1$, $S^3$, $S^7$ are the only $H$-spaces localized at $2$ by solving the famous Hopf invariant one problem \cite{Adams1}, and all odd spheres are $H$-spaces localized at any odd prime $p$ \cite{Adams2}. For $r=2$, the case $p=2$ (then the integral case) has been solved by a series of papers \cite{Adams2a}, \cite{Hubbuck}, 
\cite{Zabro1}, \cite{Douglas}, \cite{Mimura}, as well as the case $p >3$ by N. Hagelgans \cite{Nancy}. The remaining case $p=3$ is challengeable and an open question for decades, while a recent  progress on this problem can be found in \cite{Grbic}.

The phenomenon that the $H$-structures are largely controlled by the prime $p=2$ behaves similarly when we consider higher homotopy associative structures. Namely, if we consider $A_p$ spaces in the sense of J. Stasheff \cite{Stasheff}, the $A_p$-structures are controlled by that of localization at $p$, where a connected $A_2$-space is just an $H$-space. In general, For any $A_n$-space $X$, Stasheff suggests an $n$-projective space over $X$, denoted by $P_n(X)$ which is the analogy to Milnor's classifying space for topological group (See Def. \ref{defanspace} and paragraph before that for the explicit definition of $A_n$-spaces and related comments).

Let $n=p$, then it is well-known that there exists some non-trivial $p$-th power in the cohomology of $p$-stage projective space $P_p(X)$ which exactly catches the $A_p$-structure. Furthermore, Hemmi \cite{Hemmi} has defined a modified projective space $R_n(X)$ for a special family of $A_n$-spaces which is our main concern in this paper. Based on these ideas and constructions, we prove the following theorem which generalizes the result of \cite{Wilkerson1} for local spheres:
\begin{theorem}\label{main1}
Fix an odd prime $p\geq 3$ and let $X$ be a connected $p$-local $A_p$-space with cohomology ring $H^\ast(X,\mathbb{Z}/p\mathbb{Z})\cong \Lambda(x_{2m_1-1},\ldots, x_{2m_r-1})$, and $m_1\leq m_j$, $\forall j$. Define 
\begin{eqnarray*}
m={\rm g.c.d.}\{m_i~|~m_i\leq pm_1\},
\end{eqnarray*}
then $m|p-1$.
\end{theorem}

For the reverse of the theorem, we recall that in \cite{Stasheff2} Stasheff has constructed realization for polynomial algebras $\mathbb{Z}/p\mathbb{Z}[x_{2m},x_{4m},\ldots,x_{2km}]$ with $m|p-1$ using a theorem of Quillen. Here, our proof of this theorem is based on a generalization of a method of Adams and Atiyah (see \cite{Adams3} and section $2$), using which we also derive a simple proof of a finiteness theorem of Hubbuck and Mimura (\cite{Hubbuck2}, also see Theorem \ref{thmHubbuck}) which claims that there are only finitely many possible homotopy types of spaces with fixed rank $r$ which are $A_p$-spaces. 

For the special case when $p=3$, a ${\rm mod}~3$ $A_3$-space is the usual $3$-local homotopy associative $H$-space. The only simply connected homotopy associative $H$-space at $3$ of rank $1$ is $S^3$. If we define the sequence of the increasing numbers $(m_1, \ldots, m_r)$ to be the type of $X$ in Theorem \ref{main1}, then the complete list of types for rank $2$ $3$-local simply connected homotopy associative $H$-spaces are $(2,3)$, $(2,4)$, $(2, 6)$ and $(6,8)$ (see Theorem $5.1$ in \cite{Wilkerson3}). It is clear that $S^3\times S^5\stackrel{3}{\simeq} SU(3)$ provides example for $(2,3)$, $Sp(2)$ for $(2,4)$ and $G_2$ for $(2,6)$. In \cite{Harper}, Harper gives a decomposition $F_4\stackrel{3}{\simeq} K\times B_5(3)$ where $B_5(3)$ is the $S^{11}$ bundle over $S^{15}$ classified by $\alpha_1$, and further in \cite{Zabro3}, Zarbrodsky shows that $B_5(3)$ is a loop space which provides example for $(6, 8)$. 
In this paper, we consider the case of rank $3$. With the help of the method of Adams and Atiyah, and some results of Wilkerson (\cite{Wilkerson3} or Theorem \ref{mainwilk}), we prove the following theorem by careful analysis of the effect of both Steenrod operations and Adams' $\psi$-operations.
\begin{theorem}\label{main2}
Let $X$ be an indecomposable $3$-local homotopy associative $H$-space with cohomology ring $H^\ast(X, \mathbb{Z}/3\mathbb{Z})\cong \Lambda(x_{2r-1}, x_{2n-1}, x_{2m-1})$ where ${\rm deg}(x_k)=k$ and $1<r<n<m$, then the type of $X$ $(r, n,m)$ can only be one of the following:
$$(2,4,6), (2,6,8), (3,5,7), (3,6,8), (6,8,10), (6,8,12).$$
\end{theorem}

For the above list, the only known example is $Sp(3)$ which is of type $(2,4,6)$. 
Here are a few things we know about potential examples of rank $3$ $3$-local $A_3$-spaces of the remaining five types. 
For $(2,6,8)$, we can form a space $X$ as the total space of a $G_2$-principal fibration over $S^{15}$ which is classified by the generator of $\pi_{15}(BG_2)\cong\pi_{14}(G_2)\stackrel{3}{\cong}\pi_{14}(S^{3})\stackrel{3}{\cong}\mathbb{Z}/3\mathbb{Z}$. Then the classifying map factors as $S^{15}\stackrel{f}{\rightarrow} BS^3\rightarrow BG_2$ and we get $X\stackrel{3}{\simeq}(G_2\times Y)/ S^3$ where $Y$ is the total space of the fibration classified by $f$ and also an $H$-space by Theorem $7.1$ of \cite{Grbic}. However, we still do not known whether $X$ is an $H$-space or not.
For the case $(3,5,7)$, we have Nishida's $B_2^3(3)$ which is a $3$-component of $SU(7)$ (see \cite{Mimura2}). Still, we do not know whether $B_2^3(3)$ is homotopy associative. If $X$ is of type $(3,6,8)$, then $X$ has a generating complex of the form $S^5\vee A$ by the knowledge of the homotopy groups of sphere, where $A$ is of type $(6,8)$. For $(6,8,10)$, Harper and Zabrodsky have proved in \cite{Harper2} that if the exterior algebra of rank $p$ generated by $\{x_{2n-1},\mathscr{P}^1x_{2n-1},\ldots, \mathscr{P}^{p-1}x_{2n-1}\}$ can be realized by an $H$-space, then $p|n$, and the inverse question is still open for $n>p$.  For the last possible case of type $(6,8,12)$, we have $\mathscr{P}^1(x_{11})=x_{15}$ and $\mathscr{P}^3(x_{11})=x_{23}$.

The article is organized as follows. In section $2$, we will introduce a refined version of Adams and Atiyah's method in \cite{Adams3}. In section $3$, we use number theory to prove Theorem \ref{main1} and the finiteness theorem of Hubbuck and Mimura. The last section is devoted to prove Theorem \ref{main2}.

\section{A method of Adams and Atiyah}
\noindent In \cite{Adams3}, Adams and Atiyah develop a method to detect the $p$-th power of cohomology elements using Adams's $\psi$-operations. For our purpose, we need to modify it slightly.

Given a connected CW-complex $X$ with no $p$-torsion in $H^\ast(X,\mathbb{Z})$, suppose there exists a subalgebra $\bar{\mathcal{H}}$ of ${H}^\ast(X;\mathbb{Z}/p\mathbb{Z})$ such that
\begin{eqnarray*}
\bar{\mathcal{H}}\cong\bar{A}\oplus \bar{B}
\end{eqnarray*}  
as rings, $\bar{B}$ is an ideal and also $\bar{\mathcal{H}}$ and $\bar{B}$ are closed under the action of the ${\rm mod}~ p$ Steenrod algebra $\mathscr{A}_p$. Then by Atiyah-Hirzebruch-Whitehead spectral sequence and Theorem $6.5$ in \cite{Atiyah}, we have the corresponding 
filtered subalgebra $\mathcal{H}$ of $K(X)\otimes\mathbb{Z}_{(p)}$ such that 
\begin{eqnarray*}
\mathcal{H}\cong A\oplus B,
\end{eqnarray*}
as filtered rings, and also $\mathcal{H}$ and $B$ are closed under $\psi^p$-action. Write the Chern character of an element $x\in K(X)\otimes\mathbb{Z}_{(p)}$ as 
\begin{eqnarray*}
\mathrm{ch} (x)=a_0+\sum\limits_i a_{2i}+\sum\limits_{j}b_{2j}, 
\end{eqnarray*}
with $a_0\in \mathbb{Q}$, $a_{2i}\in \bar{A}^{>0}\otimes \mathbb{Q}$ and $b_{2j}\in \bar{B}^{>0}\otimes\mathbb{Q}$ (the subscripts refer to the degree), then we have 
\begin{eqnarray*}
\mathrm{ch} (\psi^k(x))=a_0+\sum\limits_i k^i a_{2i}+ \sum\limits_j k^jb_{2j}.
\end{eqnarray*}
Hence $\psi^k$ is indeed a semisimple linear transformation if we use Chern character to identify $K(X)\otimes \mathbb{Q}$ with $H^{{\rm even}}(X; \mathbb{Q})$, and the eigenspace decomposition of $\tilde{K}(X)\otimes \mathbb{Q}$ is independent of the choice of $\psi^k$. In particular, $\mathcal{H}\otimes \mathbb{Q}$ and $B\otimes \mathbb{Q}$ are invariant under $\psi^k$ for any $k$ as they are invariant under $\psi^p$, and then $\mathcal{H}$ and $B$ are also invariant under each $\psi^k$. Then the same as in \cite{Adams3}, we get (partial) eigenspace decomposition
\begin{eqnarray*}
\tilde{\mathcal{H}}\cong \oplus_{i=1}^r V_i \oplus W,
\end{eqnarray*}
\begin{eqnarray*}
B^{>0}\otimes \mathbb{Q}\cong W,
\end{eqnarray*}
where $\tilde{\mathcal{H}}=\mathcal{H}^{>0}\otimes \mathbb{Q}$, ${\rm deg}(V_{i})=2m_i$ (which means the degree of its elements) and $V_i$ is allowed to be $0$ vector space. For each $\psi^k$, $V_i$ is the eigenspace corresponding to the eigenvalue $k^{m_i}$. We also notice that $A^{>0}\otimes \mathbb{Q}\cong\oplus_{i=1}^r V_i$ but only as vector spaces.
Now define a linear  transformation on $\tilde{K}(X)\otimes \mathbb{Q}$ by
\begin{eqnarray*}
\pi_i=\prod_{j=1,\ldots, \hat{i}, \ldots, r}\frac{\psi^{k_j}-k_j^{m_j}}{k_j^{m_i}-k_j^{m_j}},
\end{eqnarray*}
 and a number 
\begin{eqnarray*}
d_i(m_1, \ldots, m_r)={\rm g.c.d.}\big\{  \prod_{j=1,\ldots, \hat{i}, \ldots, r}(k_j^{m_i}-k_j^{m_j})~|~
\forall~ {k_1,\ldots, \hat{k_i},\ldots, k_r}, k_j\in \mathbb{N}^+\big\}.
\end{eqnarray*}
Notice that $\pi_i$ induces a linear transformation $\bar{\pi}_i$ on $\oplus_{i=1}^r V_i$ which is the natural projection onto the $i$-th component $V_i$.
For any $x\in \tilde{\mathcal{H}}$, we have
\begin{eqnarray*}
\pi_i(x) \cdot\prod_{j=1,\ldots, \hat{i}, \ldots, r}(k_j^{m_i}-k_j^{m_j})=\prod_{j=1,\ldots, \hat{i}, \ldots, r}(\psi^{k_j}-k_j^{m_j})(x)\in\tilde{\mathcal{H}}.
\end{eqnarray*}
Accordingly, 
\begin{eqnarray*}
\pi_i(x) d_i(m_1,\ldots, m_r)\in \tilde{\mathcal{H}}.
\end{eqnarray*}
If we write $x=\sum\limits_i\bar{\pi}_i(x-v)+v$ for some $v\in B$, then we also have  
\begin{eqnarray*}
\bar{\pi}_i(x-v) d_i(m_1,\ldots, m_r)\in \tilde{\mathcal{H}}.
\end{eqnarray*} 
Now we make a crucial assumption that for each $i$ 
\begin{equation}\label{condition}
p^{m_i}\nmid d_i(m_1,\ldots, m_r).
\end{equation}
Since $B$ is a $\{\psi^p\}$-module, we have
\begin{eqnarray*}
\psi^p(x) 
&=& \sum\limits_i \psi^p(\bar{\pi}_i(x-v))+\psi^p(v)\\
&=& \sum\limits_i p^{m_i}\frac{\bar{\pi}_i(x-v)d_i(m_1,\ldots,m_r)}{d_i(m_1,\ldots,m_r)}+\psi^p(v)\\
&=& py+\psi^p(v)\in p\tilde{\mathcal{H}}+B,
\end{eqnarray*}
i.e., $x^p \equiv \psi^p(x) \equiv 0~{\rm mod}~(p, B)$.
Still the same as in \cite{Adams3}, $\bar{x}^p\equiv 0 ~{\rm mod}~(\bar{B})$ on the cohomology level where $\bar{x}$ denotes the corresponding element of $x$ in $\bar{\mathcal{H}}\subset H^\ast(X,\mathbb{Z}/p\mathbb{Z})$.

\begin{remark}
Notice that when $\bar{\mathcal{H}}=H^\ast(X,\mathbb{Z}/p\mathbb{Z})$ and $\bar{B}=0$, the above result is exactly the Corollary in \cite{Adams3}.  
\end{remark}
\section{Proof of Theorem \ref{main1} and the finiteness theorem}

\subsection{Proof of Theorem \ref{main1}}
\noindent We prove the theorem by contradiction, while the main task is to prove the condition (\ref{condition}) holds. So we have to do some work in number theory first.
\begin{definition} Let $n$ be a positive integer, define

1)~ $e(n)=f, \ \ \  \ \ \ \ \ {\rm if}~ n=p^f\cdot x ~{\rm and }~ p\nmid x$, 

2) ~$\nu(n)=f+1, \ \ \ {\rm if}~ n=p^f(p-1)x~{\rm and}~p\nmid x$. 
    
    ~\ \ \ \  $\nu(n)=0, \ \ \ \ \ \ \ \ {\rm if}~ p-1\nmid n$.

\end{definition}
Suppose $k$ is a primitive root modulo $p^2$, then $k$ is also a primitive root modulo $p^f$ for all $f\in \mathbb{N}^+$. Then for any positive integer $n$, we have
\begin{equation}\label{pfactor}
k^n\equiv 1 ~{\rm mod}~p^f \ \ \Longleftrightarrow \ \ n\equiv 0 ~{\rm mod}~p^{f-1}(p-1).
\end{equation}
So $\nu(n)$ is the exact exponent of $p$ in the prime factorization of $k^n-1$ if $p-1|n$.

The following lemma is well known and basic in number theory:
\begin{lemma}[Legendre, $1808$]\label{Leg1}
\begin{eqnarray*}
e(n!)=\sum\limits_{k=1}^\infty \lfloor \frac{n}{p^k} \rfloor=\frac{n-s_p(n)}{p-1},
\end{eqnarray*}
where $s_p(n)=a_k+a_{k-1}+\ldots+a_1+a_0$ is the sum of all the digits in the expansion of $n$ in the base $p$.
\end{lemma}
From above, we easily get:
\begin{corollary}\label{Leg2}
1)~$e(a!)+e(b!)\leq e((a+b)!)$;

 \qquad \qquad \qquad \ \  2)~$e((ab)!)\leq a+e(a!)$, if $b\leq p$.
\end{corollary}
Now we are ready to prove our main lemma which is a generalization of Lemma $3.5$ in \cite{Adams3}:
\begin{lemma}\label{mainlemma}
Let $p$ be an odd prime, $k$ be a primitive root modulo $p^2$, $m$, $t\in \mathbb{N}^+$ such that $m\nmid p-1$ and 
$$\prod:=\prod_{j=t,t+1\ldots,\hat{i}\ldots, tp}(k^{mi}-k^{mj}),$$
then we have 
\begin{equation}
e(\prod)<mt.
\end{equation}
\end{lemma}
\begin{proof}
We set ${\rm g.c.d.}(m, p-1)=h$, $m=ah$, and $p-1=bh$, then $a>1$ since $m\nmid p-1$. Then we have
\begin{eqnarray*}
\prod_{j=t,t+1\ldots,\hat{i}\ldots, tp}(k^{mi}-k^{mj})=\prod_{t\leq j<i}k^{mj}(k^{m(i-j)}-1) \cdot\prod_{i<j\leq tp}k^{mi}(1-k^{m(j-i)}).
\end{eqnarray*}
By (\ref{pfactor}), we only need to consider $j$'s satisfying $p-1|m(i-j)$, i.e., $b|i-j$, then we have 
\begin{eqnarray*}
e(\prod) 
&=& \prod_{t\leq j<i}e(k^{m(i-j)}-1) \cdot\prod_{i<j\leq tp}e(1-k^{m(j-i)})\\
&=& \prod_{1\leq \frac{i-j}{b}\leq \lfloor \frac{i-t}{b}\rfloor} e(k^{mb\frac{i-j}{b}}-1)  \cdot \prod_{1\leq \frac{j-i}{b}\leq \lfloor \frac{tp-i}{b}\rfloor}e(k^{mb\frac{j-i}{b}}-1)  \\
&=& \prod_{1\leq j\leq \lfloor \frac{i-t}{b}\rfloor } e(k^{mbj}-1)  \cdot \prod_{1\leq l\leq \lfloor \frac{tp-i}{b}\rfloor} e(k^{mbl}-1)  \\  
&=& \sum\limits_{1\leq j\leq \lfloor \frac{i-t}{b}\rfloor}  \nu(mbj)+    \sum\limits_{1\leq l\leq \lfloor \frac{tp-i}{b}\rfloor} \nu(mbl)\\
&=& (e(m)+1)(\lfloor \frac{i-t}{b}\rfloor+\lfloor \frac{tp-i}{b}\rfloor)+e(\lfloor\frac{i-t}{b}\rfloor !)+e(\lfloor \frac{tp-i}{b}\rfloor !)\\
&\leq& (e(m)+1)\frac{tp-t}{b}+e((\lfloor\frac{i-t}{b}\rfloor +\lfloor \frac{tp-i}{b}\rfloor )!)\\
&\leq& (e(m)+1)th+e((th)!).\\
\end{eqnarray*}
Now if $h=1$, then 
\begin{eqnarray*}
e(\prod)
&\leq& (e(m)+1)t+e(t!)\\
&=&(e(m)+1)t+\frac{t-s_p(t)}{p-1}\\
&<&t(e(m)+1+\frac{1}{p-1}).
\end{eqnarray*}
If $h\geq 2$, then
\begin{eqnarray*}
e(\prod)
&\leq& (e(m)+1)th+t+e(t!)\\
&=& (e(m)+1)th+t+\frac{t-s_p(t)}{p-1}\\
&<& t\big((e(m)+1)h+1+\frac{1}{p-1}\big).\\
\end{eqnarray*}
On the other hand, we have $a-e(m)-1\geq 1$ always holds, for otherwise, $e(a)+1=e(m)+1=a$ implies $a=1$ (we use $p\geq 3$ here). Now combining all above, it is easy to see $e(\prod)<mt$ in both cases.
\end{proof}

Now we are going to prove Theorem \ref{main1}. First we recall some background on $A_n$-spaces for which Stasheff's original papers \cite{Stasheff} are the standard reference. Stasheff's $A_n$-spaces can be defined inductively with the help of Stasheff polytopes which are also called associahedrons. Explicitly, an associahedron $K_n$ is an $(n-2)$-dimensional convex polytope whose vertices are in one to one correspondence with the parenthesizings of the word $x_1x_2\ldots x_n$, and whose edges correspond to single application of the associativity rule. In particular, $K_2$ is a point, $K_3$ is a interval and $K_4$ is the convex hull of a pentagon. There are canonical maps among $K_n$'s. Indeed, the family $\mathcal{K}=\{K_n\}$ can be endowed with an operadic structure such that any $\mathcal{K}$-space is the so-called $A_\infty$-space ($\mathcal{K}$ is called $A_\infty$-opeard). Then an $A_n$-space is just an space with the action of $\mathcal{K}$ only up to the $n$-stage (the corresponding opeard is called $A_n$-operad). Stasheff also gave another equivalence description of $A_n$-spaces which he used as definition.

\begin{definition}[Definition $1$ in \cite{Stasheff}]\label{defanspace}
An $A_n$-structure on a space $X$ consists of an $n$-tuple of maps 
\begin{equation*}
 \xymatrix{
 X  \ar@{=}[r]   &  E_1 \ar@{^{(}->}[r] \ar[d]^{p_1}  &E_2\ar@{^{(}->}[r] \ar[d]^{p_2}   & \cdots \ar@{^{(}->}[r]    & E_n \ar[d]^{p_n}\\
\ast \ar@{=}[r]  & B_1 \ar@{^{(}->}[r]                          &B_2\ar@{^{(}->}[r]                            & \cdots \ar@{^{(}->}[r]    &B_n 
 }
\end{equation*}
such that each $p_i$ is a quasi-fibration, and there is a contracting homotopy $h: C E_{n-1}\rightarrow E_n$ such that $h(CE_{i-1})\subset E_i$.
\end{definition}
Note that if $A_n$-structure is given by the operadic action, the above diagram can be constructed such that $B_i$ is the $i$-th `projective space'  $P_n(X)$ over $X$ (like Milnor's construction). The reverse process was done by Stasheff. The projective space is crucial for there are non-trivial $n$-th powers in its cohomology ring. 

Here, the key construction for our proof of Theorem \ref{main1} is the so-called modified projective space due to Hemmi \cite{Hemmi2} which is an analogy of Stasheff's $n$-projective space \cite{Stasheff}. Since we will not use the explicit construction of this concept, we only recall some properties stated in the following lemma.
\begin{lemma}[part of Theorem $1.1$ in \cite{Hemmi2}]\label{lemma1}
Let $n\geq 3$ and $X$ be a finite $A_n$-space with cohomology ring 
\[H^\ast(X,\mathbb{Z}/p\mathbb{Z})\cong \Lambda(x_{2m_1-1},\ldots, x_{2m_r-1}), \ \ {\rm deg}(x_{2m_i-1})=2m_i-1,\]
then there exists a \textit{modified projective space} $R_n(X)$ with a map $\varepsilon: \Sigma X \rightarrow R_n(X)$ such that 
\begin{eqnarray*}
\bar{\mathcal{H}}\cong \bar{A}\oplus \bar{B}=\mathbb{Z}/p\mathbb{Z}~[y_{2m_1},\ldots, y_{2m_r}] /({\rm height}~ n+1)\oplus \bar{B}
\end{eqnarray*}
as rings for some subalgebra $\bar{\mathcal{H}}$ of $H^\ast(R_n(X),\mathbb{Z}/p\mathbb{Z})$ 
and $\varepsilon^\ast(y_{2m_i})=\sigma^\ast(x_{2m_i-1})$, where the idea under quotient in the first factor is generated by monomials of length greater or equal to $n+1$. Further $\bar{\mathcal{H}}$ and $\bar{B}$ are closed under the action of the ${\rm mod}~p$ Steenrod algebra $\mathscr{A}_p$.
\end{lemma}

Now we are ready to prove Theorem \ref{main1}.

\noindent{\bf Proof of Theorem \ref{main1}}\quad We prove the theorem by contradiction, and assume $m\nmid p-1$.
 By Lemma \ref{lemma1}, $H^\ast(R_p(X))$ contains a truncated polynomial algebra 
 \begin{equation*}
 \mathbb{Z}/p\mathbb{Z}[y_{2m_1},\ldots,y_{2m_r}]/({\rm height}~p+1)\hookrightarrow H^\ast(R_p(X)). 
 \end{equation*}
 Let us denote $Y(X)=R_p^{2pm_1+1}(X)$ to be the $(2pm_1+1)$-skeleton of $R_p(X)$, we then have a ring decomposition $i^\ast(\bar{\mathcal{H}})\cong i^\ast(\bar{A})\oplus i^\ast(\bar{B})$ where $i: Y(X)\hookrightarrow R_p(X)$ is the canonical inclusion. Then $y_{2m_1}^p\not\equiv 0 ~{\rm mod}~(i^\ast(\bar{B}))$. We then set $m_i=ms_i$, and apply Lemma \ref{mainlemma} for $t=s_1$ and $m=m$ since $m\nmid p-1$ by assumption. Then we get $e(\prod)<ms_1=m_1$, which implies the condition (\ref{condition}) holds for $Y(X)$ since $m_1$ is the lowest degree. Further $i^\ast(\bar{\mathcal{H}})$ and $i^\ast(\bar{B})$ are closed under the action of $\mathscr{A}_p$, hence by the argument in Section $2$, $\bar{x}^p\equiv 0 ~{\rm mod}~(i^\ast(\bar{B}))$ for any $\bar{x}\in i^\ast(\bar{\mathcal{H}})$ which contradicts the fact $y_{2m_1}^p\not\equiv 0 ~{\rm mod}~(i^\ast(\bar{B}))$ and the proof of Theorem \ref{main1} is completed.

{\protect\vspace{-2pt}\rightline{$\square$}}

\subsection{The Finiteness Theorem for finite $A_p$-spaces}
As another application, we prove the following theorem of Hubbuck and Mimura:
\begin{theorem}[\cite{Hubbuck2}]\label{thmHubbuck}
Let $X$ be a connected finite ${\rm mod}~p$ $A_p$ space of rank $r$. Then there are only finitely many possible homotopy types for the space $X$.
\end{theorem} 
\begin{proof}
Suppose $X$ has the type $(m_1, m_2,\ldots, m_r)$ with $m_1\leq m_2\leq\ldots \leq m_r$ and form the space $Y(X)=\frac{R_p^{2pm_r+1}(X)}{R_p^{2m_r-1}(X)}$ which is the $2pm_r+1$-skeleton of $R_p(X)$ with the $2m_r-1$-skeleton pinched to a point. As in the proof of Theorem \ref{main1}, we can get a ring decomposition $p^{\ast-1}i^\ast(\bar{\mathcal{H}})\cong p^{\ast-1}i^\ast(\bar{A})\oplus p^{\ast-1}i^\ast(\bar{B})$ using the canonical inclusion and projection, such that $p^{\ast-1}i^\ast(\bar{\mathcal{H}})$ and $p^{\ast-1}i^\ast(\bar{B})$ are closed under the action of $\mathscr{A}_p$, and $y_{2m_r}^p$ is nontrivial module $p^{\ast-1}i^\ast(\bar{B})$.
We may also fix a number $N(p ,r)$ only depending on $p$ and $r$ such that $N(p, r)\geq {\rm dim}~p^{\ast-1}i^\ast(\bar{\mathcal{H}})$, and notice that the largest difference of the degrees of any two elements in $p^{\ast-1}i^\ast(\bar{\mathcal{H}})$ is bounded by $2(p-1)m_r$. Suppose the even part of $p^{\ast-1}i^\ast(\bar{\mathcal{H}})$ concentrates in dimension $2t_1, 2t_2, \ldots$, then for sufficient large $m_r$ we have 
\begin{eqnarray*}
e\big(\prod_{j\neq i}(k^{t_i}-k^{t_j})\big)
&\leq& \sum\limits_{j\neq i}\big(e(t_i-t_j)+1\big)\\
&\leq&N(p,r)\lfloor{\rm log}_p(2(p-1)m_r)\rfloor+N(p,r)\\
&<&m_r,
\end{eqnarray*}
for any $i$, i.e., the condition (\ref{condition}) holds which contradicts the existence of the nontrivial $p$-th power in $p^{\ast-1}i^\ast(\bar{\mathcal{H}})$. Accordingly the largest dimension of the generators is bounded and there are only finite possible types for $X$. Also by Corollary $4.2$ in \cite{Body}, there are only finite homotopy types for each certain type. Then in all there are finite homotopy types for fixed rank and the theorem has been proved. 
\end{proof}

\section{rank $3$ ${\rm mod}~3$ homotopy associative $H$-spaces}
\noindent For rank $3$ ${\rm mod}~3$ homotopy associative $H$-spaces, we will consider Stasheff's $3$-projective space instead of Hemmi's modified projective space used in the proof of Theorem \ref{main1}. The key lemma as an analogy to Lemma \ref{lemma1} for projective spaces is the following well-known result.
\begin{lemma}[e.g. \cite{Iwase}]{\label{lemma1a}}
Let $n\geq 3$ and $X$ be a finite $A_n$-space with cohomology ring 
\[H^\ast(X,\mathbb{Z}/p\mathbb{Z})\cong \Lambda(x_{2m_1-1},\ldots, x_{2m_r-1}), \ \ {\rm deg}(x_{2m_i-1})=2m_i-1,\]
 such that each $x_{2m_i-1}$ is $A_n$-primitive, i.e., $x_{2m_i-1}$ lies in the image of a series of natural morphisms: 
 \[H^\ast (P_n(X))\rightarrow H^\ast(P_{n-1}(X))\rightarrow \cdots \rightarrow H^\ast(P_1(X)=\Sigma X)\stackrel{\cong}{\leftarrow} H^{\ast-1}(X),\]
 then we have ring isomorphism
\begin{eqnarray*}
H^\ast(P_n(X), \mathbb{Z}/p\mathbb{Z})\cong A\oplus B=\mathbb{Z}/p\mathbb{Z}~[y_{2m_1},\ldots, y_{2m_r}] /({\rm height}~ n+1)\oplus B
\end{eqnarray*}
as $\mathscr{A}_p$-modules and $A^{+}\cdot B=0$, where ${\rm deg}(y_{2m_i})=2m_i$.
\end{lemma}   
Notice that the corresponding result in the context of $K$-theory can be easily deduced, and for rank $3$ ${\rm mod}~3$ homotopy associative $H$-spaces, the primitive assumption is automatically satisfied. To prove Theorem \ref{main2}, we will also use the following theorem of Wilkerson.
\begin{theorem}[Theorem $6.1$ and $6.2$ in \cite{Wilkerson3}]\label{mainwilk}
Let $X$ be a finite ${\rm mod}~p$ $A_p$-space with cohomology ring $H^\ast(X,\mathbb{Z}/p\mathbb{Z})\cong \Lambda(x_{2m_1-1},\ldots, x_{2m_r-1})$, such that $m_1\leq m_2\leq \ldots \leq m_r$ and $m_r>p$, then 

(1)~there is a $x_{2m_k-1}$ with $m_r-m_k=s(p-1)$ for some $1\leq s \leq e(m_r)+1$.

(2)~if $p\nmid m_i$ for some $i$, there is a $x_{2m_j-1}$ such that $m_j=k_jm_i-p+1$ for some $1\leq k_j \leq p$.
\end{theorem}

Combining Theorem \ref{main1} and Theorem \ref{mainwilk}, we are left to consider the following $4$ cases for the possible types of the ${\rm mod}~3$ $A_3$-space $X$ in Theorem \ref{main2}: 

\textbf{Case} $\mathbf{1}$ \ \ \ \ $3|m$, $3|n$ and $m-n=2s$ with $1\leq s \leq e(m)+1$,

\textbf{Case} $\mathbf{2}$ \ \ \ \ $3|m$, $3\nmid n$ and $m-n=2s$ with $1\leq s \leq e(m)+1$,

\textbf{Case} $\mathbf{3}$ \ \ \ \ $3\nmid m$ and $m-n=2s$ with $1\leq s\leq e(m)+1$,

\textbf{Case} $\mathbf{4}$  \ \ \ \ $m-r=2t$ with $1\leq t\leq e(m)+1$, and $m-n\neq 2s$ for any $s$ such that $1\leq s\leq e(m)+1$.

For \textbf{Case} $\mathbf{1}$, we need the following lemma:
\begin{lemma}\label{lemmacase1}
Under the condition of Theorem \ref{main2} and \textbf{Case} $\mathbf{1}$, we have 

(1)~ If $r=2$, $m>n>6$ and $e(m)\geq e(n)+2$, then $8e(n)+23\geq n$;

(2)~If $r=2$, $m>n>6$ and $e(m)=e(n)+1$, then $8{\rm max}\big( e(3n-m), e(3n-2m)\big)+15\geq n$;

(3)~If $m\leq 3r$, $e(m)\geq e(n)+2$, then $7e(n)+\lfloor {\rm log}_3(m-r)\rfloor+24\geq m$ or $8\lfloor {\rm log}_3(m-r)\rfloor+24\geq 3r$;

(4)~If $m\leq 3r$, $e(m)=e(n)+1$, then $7{\rm max}\big( e(3n-m), e(3n-2m)\big)+\lfloor {\rm log}_3(m-r)\rfloor+17\geq m$ or $8\lfloor {\rm log}_3(m-r)\rfloor+24\geq 3r$.
\end{lemma}
\begin{proof}
By the condition, we have a $\{\psi^k\}$-module $K=\mathbb{Z}_{(3)}[x_r, x_n, x_m]/({\rm height} ~4)$ where the subscripts refer to the filtration degree. For $(1)$ and $(2)$, $r=2$, and we only need to consider $K^\prime= K-\{x_r^i~|~i=1, 2, 3\}$. We can set 
$$S=\big\{ 2i+jn+km~|~ (i, j, k)\neq (i,0,0), 0\leq|j|,|k|\leq 3, 0\leq |i|\leq 2 \big\},$$
and define $\Phi(i, j, k)=|2i+jn+km|$. Now For $(1)$ we have $e(\Phi(0, j, k))\leq e(n)+1$ and $e(\Phi(i, j, k))=0$ if $|i|=1$, or $2$. And we notice that there are $9$ elements of the form $x_n^\ast x_m^\ast$, $5$ elements of the form $x_r^1x_n^\ast x_m^\ast$ and $2$ elements of the form $x_r^2x_n^\ast x_m^\ast$ in $K^\prime$, then
\begin{eqnarray*}
e\big(\prod_{(\tilde{i}, \tilde{j}, \tilde{k})\neq (0, j ,k)}(2^{jn+km}-2^{2\tilde{i}+\tilde{j}n+\tilde{k}m})\big)
&\leq&\sum\limits e\big(\Phi(-\tilde{i}, j-\tilde{j}, k-\tilde{k})\big)+ 15\\
&\leq& 8(e(n)+1)+15\\
&=& 8e(n)+23.\\
\end{eqnarray*}
Similarly, we have $e\big(\prod_{(\tilde{i}, \tilde{j}, \tilde{k})\neq (1, j ,k)}\big)\leq 4e(n)+19$ and $e\big(\prod_{(\tilde{i}, \tilde{j}, \tilde{k})\neq (2, j ,k)}\big)\leq e(n)+16$. Since condition (\ref{condition}) should fail for $X$, we must have $8e(n)+23\geq n$. 

The remaining three claims can be proved similarly, and notice that for $(3)$ and $(4)$, we work with $K^\prime =K-\{x_r, x_n\}$ if $m\leq 2r$ and with $K^\prime=K-\{x_r,x_n,x_r^2\}$ if $m>2r$.
\end{proof}
Now we are ready to deal with \textbf{Case} $\mathbf{1}$:
\begin{proposition}\label{prop1}
Under the condition of Theorem \ref{main2} and \textbf{Case} $\mathbf{1}$, the only possible types of $X$ are:
$$(2, 3, 9), (2, 12,18), (2, 21, 27), (2, 30, 36), (2, 39, 45),$$
$$(7, 12,18), (10, 12,18), (16, 30, 36), (19, 30, 36).$$
\end{proposition}
\begin{proof}
By Theorem \ref{main1}, we have ${\rm g.c.d.}(r, n, m)\leq 2$, then $3\nmid r$. So by Theorem \ref{mainwilk}, we have $x=\lambda r-2$ with $\lambda\in \{1, 2, 3\}$ and $x\in\{ r, n, m\}$. Then either $r=2$ or $n=2r-2$ or $m=2r-2$. 

We prove the proposition under condition $e(m)>e(n)$ first:

(1)~If $r=2$, $n>6$ and $e(m)\geq e(n)+2$, by Lemma \ref{lemmacase1}, we have $8e(n)+23\geq n$. Then \begin{eqnarray*}
3^{e(m)}\cdot f
&=&m=n+2s\\
&\leq& 8e(n)+23+2(e(n)+1)\\
&=&10e(n)+25\\
&\leq& 10e(m)+5.
\end{eqnarray*} 
Since $e(m)\geq 3$, we have 
$m=27$ and $e(m)=3$. Then $e(n)=e(s)=1$ and $n$ is odd. Now it is not hard to check that $(2, 21, 27)$ is the only possible type satisfying all the conditions.

(2)~If $r=2$, $n>6$ and $e(m)=e(n)+1$, by Lemma \ref{lemmacase1}, $8{\rm max}\big( e(3n-m), e(3n-2m)\big)+15\geq n$. If $8e(3n-m)+15\geq n$, then $8e(n-s)+12\geq 8e(n-s)+15-s\geq n-s$ for $e(n-s)=e(2(n-s)=3n-m)\geq e(m) \geq 2$ and $s\geq 3$. Then it is to show $n-s=9,18$ or $27$. In any case, $s\leq e(m)+1\leq 4$ which implies $s=3$. And then $m-n=6$ and $n=12, 21$ or $30$. But since $e(m)=e(n)+1=2$, only $(2, 12, 18)$ or $(2,30,36)$ is possible for our $X$.

If $8e(3n-2m)+15\geq n$, then $8e(n-4s)+3\geq 8e(n-4s)+15-4s\geq n-4s$ for $n-4s=3n-2m$ and $s\geq 3$. Then we get $n-4s=9, 18$ or $27$. Again since $e(n-4s)\geq e(m)\geq 2$ and $s\leq e(m)+1$, we have $s=3$. Then $m-n=6$ and $n=21, 30$ or $39$ and only $(2,30,36)$ and $(2,39,45)$ survive. 

(3)~If $m\leq 3r$, $e(m)\geq e(n)+2$, by Lemma \ref{lemmacase1} we have $7e(n)+\lfloor {\rm log}_3(m-r)\rfloor+24\geq m$ or $8\lfloor {\rm log}_3(m-r)\rfloor+24\geq 3r$. We also notice that $r\neq 2$ which by our early discussion implies $n=2r-2$ or $n=2m-2$. If the first inequality and $n=2r-2$ hold, then
\begin{eqnarray*}
2r-2=n<m
&\leq& 7e(n)+\lfloor {\rm log}_3(m-r)\rfloor+24\\
&\leq& 7e(r-1)+\lfloor {\rm log}_3(2r)\rfloor+24\\
&\leq&8\lfloor {\rm log}_3r\rfloor+25,
\end{eqnarray*}
which implies $r\leq 21$. Then $m\leq 3r\leq 63$ implies $m=27$ or $54$ for $e(m)\geq 3$. So $e(m)=3$ and $e(n)=1$. Since $3|s$ and $s\leq e(m)+1$, we have $s=3$ and $m-n=6$. Then we see $m=27$ is impossible for $n$ is even, while $m=54$ leads to $r=25$ which contradicts our previous calculation. 
Similar arguments can be applied to other $3$ cases which will show there are no types left.

(4)~If $m\leq 3r$, $e(m)=e(n)+1$, by Lemma \ref{lemmacase1} and similar calculations as in $(3)$, we get $(r, n, m)=(7, 12,18)$, $(10, 12,18)$, $(16, 30, 36)$ or $(19, 30, 36)$.

(5)~By Theorem \ref{main1}, the only remaining case under condition $e(m)>e(n)$ is $n\leq  3r$ but $m>3r$. If $r=2$, then $n=3$ or $6$ which gives $(r, n, m)=(2, 3, 9)$. When $n=2r-2$,  we have $m/3+2<m-n=2s\leq 2e(m)+2$ which is impossible. Further, $m=2r-2$ can not hold by our assumption.

We have proved the proposition when $e(m)>e(n)$. If $e(n)\geq e(m)$, then $e(s)=e(m-n)\geq e(m)\geq s-1\geq 0$ which implies $s=1$ and $m-n=2$. However, since $3|m$ and $3|n$, this is impossible.
\end{proof}

For the remaining cases, we will also use a theorem of Hemmi:
\begin{theorem}[Theorem $1.2$ in \cite{Hemmi}, also see Section $8$ in \cite{Hemmi2}]\label{Hemmi}
Let $X$ be a homotopy $H$-space with $H^\ast(X; \mathbb{Z}/3\mathbb{Z})$ being finite. Then for any $n\in \mathbb{Z}$ with $n\not\equiv 0 ~{\rm mod}~3$ and $n>3$, if 
\begin{equation}\label{Hemmicondition}
QH^{2(3^a\cdot2t)-1}(X,\mathbb{Z}/3\mathbb{Z})=0,~{\rm for}~ t\geq n-1,
\end{equation}
then we have 
\begin{equation}
\mathscr{P}^{3^a}: QH^{2(3^a (n-2))-1}(X, \mathbb{Z}/3\mathbb{Z})\rightarrow QH^{2(3^an)-1}(X, \mathbb{Z}/3\mathbb{Z})
\end{equation}
is an epimorphism, where $QH^{\ast}(X, \mathbb{Z}/3\mathbb{Z})=H^{\ast}(X, \mathbb{Z}/3\mathbb{Z})/DH^{\ast}(X, \mathbb{Z}/3\mathbb{Z})$ and $DH^{\ast}(X, \mathbb{Z}/3\mathbb{Z})$ is the submodule consisting of decomposable elements.
\end{theorem}

\begin{proposition}\label{prop2}
Under the condition of Theorem \ref{main2} and \textbf{Case} $\mathbf{2}$, the only possible types of $X$ are:
$$(2, 4, 6), (3, 4,6), (3, 5, 9), (6,8,12).$$

\end{proposition}
\begin{proof}
Since $3\nmid n$, by Theorem \ref{mainwilk}, we have $x=\lambda n-2$ with $x$ and $\lambda$ as before. Then either $r=n-2$ or $m=2n-2$.

(1)~ $r=n-2$. If $m>3r$, $m-n>m-(m/3+2)=2/3m-2$. So we have $2/3m-2=2s<2e(m)+2$ which implies $m=9$. Then $(r,n.m)=(2,4,9)$ contradicts the fact that $m-n$ is even. 

If $2n-2=2r+2\leq m\leq 3r$, then $m/2-1\leq m-n=2s\leq 2e(m)+2$ which implies $(r,n,m)=(2,4,6)$ or $(3, 5, 9)$. 

If $m<2n-2$ and $n=3k+2$ for some $k$, then in the $\mathscr{A}_p$-module $\bar{K}=\mathbb{Z}/3\mathbb{Z}[x_r, x_n, x_m]/({\rm height} ~4)$, $\mathscr{P}^1(x_r)=c x_n$ with $c\not\equiv 0~{\rm mod}~3$ by Theorem \ref{Hemmi}. By the Adem relation
\begin{equation}\label{Adem1}
\mathscr{P}^1\mathscr{P}^3\mathscr{P}^{3k-1}=\epsilon \mathscr{P}^1\mathscr{P}^{3k+2}+2\mathscr{P}^{3k+2}\mathscr{P}^1,
\end{equation}
we have $\mathscr{P}^{3k-1}(x_r)\neq 0$ which implies $9k-2=3n-8$ has to be the degree of some monomial in $K$. Then by direct computation, we get $n=8$ and $r=6$ which implies $m<14$. Since $3|m$, we have $m=9$ or $12$. When $m=9$, $m-n=1$ is odd which is impossible. So we have $(r,n,m)=(6,8,12)$.

 If $m<2n-2$ and $n=3k+1$, then $r=3k-1$ which by Theorem \ref{mainwilk} implies $x=\lambda r-2$ with $x\in \{r, n, m\}$ and $\lambda\in \{1,2,3\}$. Then we have $r=2$ or $n=2r-2$ either of which is impossible.

(2)~$m=2n-2$. We have $m/2-1= m-n=2s\leq 2e(m)+2$ which implies $(r,n,m)=(2,4,6)$ or $(3,4,6)$.
\end{proof}

\begin{proposition}\label{prop3}
Under the condition of Theorem \ref{main2} and \textbf{Case} $\mathbf{3}$, the only possible types of $X$ are:
$$(2,3,5), (2,6,8), (3,5,7), (3,6,8), (4,6,8), (5,6,8) , (6,8,10),$$
$$(8,12,14), (12, 18, 20), (18,24,26), (21,27,29), (30,36,38).$$
\end{proposition}
\begin{proof}
Since $3\nmid m$, we have $m-n=2$. Then by Theorem \ref{Hemmi}, we have $\mathscr{P}^1(x_n)\neq 0$.

(1) If $m=3k+1$, we have $n=3k-1$ which by Theorem \ref{mainwilk} implies $x=\lambda n-2$ as before. Then either $r=n-2$, or $m=2n-2$ or $m=3n-2$, while the last two cases are easy to be checked and are impossible. For $r=n-2$, we apply Theorem \ref{Hemmi} to get $\mathscr{P}^1(x_r)\neq 0$ and again by Adem relation (\ref{Adem1}), we get $\mathscr{P}^{r-1}(x_r)\neq 0$ which implies $(r,n,m)=(3,5,7)$ or $(6,8,10)$.

(2) If $m=3k+2$, again by Adem relation (\ref{Adem1}) we have $\mathscr{P}^{n-1}(x_n)\neq 0$. By comparing the degree and application of Theorem \ref{mainwilk}, we get a list of possible types: $(2,3,5)$, $(2,6,8)$, $(3,6,8)$, $(4,6,8)$, $(5,6,8)$ $(8,12,14)$ and also a special type $(r, r+6, r+8)$ with $3|r$. For this remaining case, if $r=3l$ with $l\not\equiv 1~{\rm mod}~3$, then Theorem \ref{Hemmi} implies $\mathscr{P}^3(x_r)\neq0$. By the Adem relation 
\begin{equation}\label{Adem2}
\mathscr{P}^9\mathscr{P}^{3l-1}=\epsilon_1 \mathscr{P}^{3l+8}+\epsilon_2 \mathscr{P}^{3l+7}\mathscr{P}^{1}+\epsilon_3 \mathscr{P}^{3l+6}\mathscr{P}^{2}+\mathscr{P}^{3l+5}\mathscr{P}^{3},
\end{equation}
we have $\mathscr{P}^{3l-1}(x_r)\neq 0$ which gives $(r,n,m)=(18,24,26)$.

For $l\equiv 1~{\rm mod}~3$, we argue similarly as in Lemma \ref{lemmacase1} to get a condition $m\leq 44$. Then the possible types are $(12, 18, 20)$, $(21,27 ,29)$ and $(30,36,38)$.
\end{proof}

\begin{proposition}\label{prop4}
Under the condition of Theorem \ref{main2} and \textbf{Case} $\mathbf{4}$, the only possible types of $X$ are:
$$(2,3,4), (2,3,6).$$
\end{proposition}
\begin{proof}
If $m>3r$, then $2t=m-r>2r$, i.e., $r<t$. Then we have $m=r+2t<3t\leq 3e(m)+3$ which is impossible. So we have $m\leq 3r$. 

If $3\nmid m$, then $m-r=2$ and $(r,n,m)=(r,r+1,r+2)$. Further, if $3|r$, then $3\nmid n$ which implies $x=\lambda n-2$ as usual.  However it is easy to check the later is impossible. Then we get $3\nmid r$ which implies $x=\lambda r-2$. At this case, the only possible type is $(r,n,m)=(2,3,4)$.

Now suppose $3|m$. If $3\nmid r$, we have $r=2$, $n=2r-2$, $n=3r-2$ or $m=2r-2$ by Theorem \ref{mainwilk}. When $r=2$, we get $(r,n,m)=(2,3,6)$ while $(2,5,6)$ is impossible since $\lambda 5-2\in \{3,8, 13\}$. When $n=2r-2$, $r=n/2+1<m/2+1$. Then $m/2-1<m-r=2t\leq 2e(m)+2$ which implies $m=6$ or $9$. For $n$ is even, $n=4$ when $m=6$ which implies $r=3$. But $3\nmid r$, so $m=6$ is impossible. If $m=9$, then we have $(r,n,m)=(4,6,9)$ or $(5,8,9)$ either of which is impossible since $9-4\neq 2t$ and $\lambda 8-2\in\{6,14,22\}$. The other two cases can be treated similarly and lead to no possible types.

 If $3|r$, then $3\nmid n$ which implies $r=n-2$ or $m=2n-2$. When $r=n-2$, we argue exactly the same as in the proof of the first case in Proposition \ref{prop2} and get none of types is possible in this case. When $m=2n-2$, we see $r<n=m/2+1$ which implies $m/2-1<m-r=2t\leq 2e(m)+2$. Again, no types survive. 
\end{proof}

We recall the following theorem of Wilkerson and Zabrodsky \cite{Wilkerson2} which is also reproved by McCleary \cite{Mccleary}, and later strengthened by Hemmi in \cite{Hemmi3} where the assumption of the primitivity of the generators has been removed:
\begin{theorem}\label{quasiregular}
Let $X$ be a simply connected ${\rm mod}~p$ $H$-space with cohomology ring $H^\ast(X,\mathbb{Z}/p\mathbb{Z})=\Lambda(x_{2m_1-1},\ldots, x_{2m_r-1})$, and $m_1\leq m_2 \leq \ldots \leq m_r$. If $m_r-m_1< 2(p-1)$, then $X$ is $p$-quasi-regular, i.e., $X$ is $p$-equivalent to a product of odd spheres and $B_n(p)$'s, where $B_n(p)$ is the $S^{2n+1}$-fibration over $S^{2n+1+2(p-1)}$ characterized by $\alpha_p$.
\end{theorem}

Now we are ready to prove Theorem \ref{main2}.

\noindent{\bf Proof of Theorem \ref{main2}}\quad  We collect all the types obtained from Proposition \ref{prop1}, \ref{prop2}, \ref{prop3} and \ref{prop4}, and prove the theorem case by case.

First, we notice that $(2,3,4)$, $(2,3,5)$, $(3,4,6)$ and $(5,6,8)$ are quasi-regular by Theorem \ref{quasiregular}. 

If $(r,n,m)=(4,6,8)$, we already know $\mathscr{P}^1(x_n)=x_m$ in 
$\bar{K}=\mathbb{Z}/3\mathbb{Z}[x_r, x_n, x_m]/$ $({\rm height} ~4)$. Then by degree reason 
$$\mathscr{P}^4(x_r)=\mathscr{P}^1\mathscr{P}^3(x_r)=\mathscr{P}^1(\lambda x_rx_n)=\lambda \mathscr{P}^1(x_r)x_n+\lambda x_rx_m,$$
 which contradicts that 
$\mathscr{P}^4(x_r)=x_r^3$. So $(4,6,8)$ cannot be the type of $X$.

If $(r,n,m)=(3,5,9)$, we still have $\mathscr{P}^1(x_r)=x_n$ by Theorem \ref{Hemmi}. Then by Adem relation \ref{Adem1}, we have $\mathscr{P}^2(x_r)\neq 0$ which is impossible since $K_{7}=0$.

If $(r,n,m)=(8,12,14)$, we know $\mathscr{P}^1(x_n)=x_m$ in $\bar{K}$. Then by degree reason we have
$$
2\mathscr{P}^8(x_r)
=\mathscr{P}^1\mathscr{P}^1\mathscr{P}^6(x_r)
=\mathscr{P}^1\mathscr{P}^1(\lambda x_rx_n)
=\lambda x_r\mathscr{P}^1(x_m),
$$
which implies $\mathscr{P}^1(x_m)=\mu x_r^2$ with $3\nmid \mu$. On the other hand, we have $\mathscr{P}^{11}(x_n)\neq 0$ from the proof of Proposition \ref{prop3}, which implies $\mathscr{P}^1\neq 0: \bar{K}_{30}=\mathbb{Z}/p\mathbb{Z}(x_r^2x_m)\rightarrow \bar{K}_{32}$. But $\mathscr{P}^1(x_r^2x_m)=x_r^2\mathscr{P}^1(x_m)=0$ and then $(r,n,m)=(8,12,14)$ is impossible.

If $(r,n,m)=(10,12,18)$, we have $\mathscr{P}^1\mathscr{P}^9(x_r)=\mathscr{P}^{10}(x_r)=x_r^3$ which implies $\mathscr{P}^1(x_rx_m)=x_r\mathscr{P}^1(x_m)+\mathscr{P}^1(x_r)x_m=\lambda x_r^3$ with $3\nmid \lambda$. Then we have $\mathscr{P}^1(x_r)=0$ and $\mathscr{P}^1(x_m)=\lambda x_r^2$. Then by Adem relation
\begin{equation}\label{Adem3}
\mathscr{P}^3\mathscr{P}^7=-\mathscr{P}^{10}+\mathscr{P}^9\mathscr{P}^1,
\end{equation} 
we have $\mathscr{P}^3(x_n^2)=\mu x_r^3$ with $3\nmid \mu$. However, $\mathscr{P}^3(x_n^2)=2x_n\mathscr{P}^3(x_n)$ is not equal to $\mu x_r^3$, so $(10,12,18)$ can not be the type of $X$.

If $(r,n,m)=(12,18,20)$, we have $\mathscr{P}^1(x_n)=x_m$. Again, by Adem relation \ref{Adem1}, we have 
$\mathscr{P}^{17}(x_n)\neq 0$ and $\mathscr{P}^3\neq 0: \bar{K}_{52}=\mathbb{Z}/p\mathbb{Z}(x_rx_m^2)\rightarrow \bar{K}_{58}=\mathbb{Z}/p\mathbb{Z}(x_nx_m^2)$ which implies $\mathscr{P}^3(x_r)={x_n}$. 
However the Adem relation
\begin{equation}\label{Adem5}
\mathscr{P}^3\mathscr{P}^9=\mathscr{P}^{12}+\mathscr{P}^{11}\mathscr{P}^1
\end{equation}
implies $\mathscr{P}^3(x_rx_n)=\pm x_r^3$ which contradicts $\mathscr{P}^3(x_rx_n)=x_r\mathscr{P}^3(x_n)+\mathscr{P}^3(x_r)x_n=x_r\mathscr{P}^3(x_n)+x_n^2$. So $(r,n,m)=(12,18,20)$ is impossible.

For $(r,n,m)=(2, 12, 18)$, or $(7,12,18)$, we first prove the following lemma:

\begin{lemma}
Let $X$ be a $p$-local $A_p$-space with cohomology ring $H^\ast(X,\mathbb{Z}/p\mathbb{Z})\cong \Lambda(x_{2m_1-1},\ldots, x_{2m_r-1})$, such that each $x_{2m_i-1}$ is $A_p$-primitive, $m_1\leq m_j$, $\forall j$ and $p<m_r$. Then there is a $x_{2m_k-1}$ such that $\mathscr{P}^i(x_{2m_k-1})=x_{2m_r-1}$ for some suitable nonzero $i$. 
\end{lemma}
\begin{proof}
This is essentially Lemma $4.4$ in \cite{Wilkerson3}, which claims that in the $\{\psi^p\}$-submodule $K= \mathbb{Z}_{(p)}[x_{m_1},\ldots, x_{m_r}]/({\rm height}~p+1)$ of $K(P_p(X))\otimes \mathbb{Z}_{(p)}$, there is a $x_{m_k}$ such that $\psi^p(x_{m_k})=\lambda x_{m_r}+{\rm else}$ with $\lambda\neq0$, for in Theorem $6.5$ of \cite{Atiyah}, Atiyah has shown that if $\psi^p(x_q)=\sum\limits_i p^{q-i}x_i$, then $\mathscr{P}^i(\bar{x}_q)=\bar{x}_i$ holds on the cohomology level.
\end{proof}

Now we return to the proof Theorem \ref{main2}. Using above lemma, we see $\mathscr{P}^3(x_{12})=x_{18}$ holds in $\bar{K}\subset H^\ast(P_3(X))$ for both mentioned cases. Then we apply Adem relation (\ref{Adem5}) to $x_{12}$. Since in both cases $\mathscr{P}^{11}\mathscr{P}^1(x_{12})=0$, we have $\mathscr{P}^3\mathscr{P}^9(x_{12})=\pm x_{12}^3$. However, $\bar{K}_{30}=\mathbb{Z}/p\mathbb{Z}(x_{12}x_{18})$, and since $\bar{K}$ is truncated, $\mathscr{P}^3(x_{12}x_{18})=x_{12}\mathscr{P}^3(x_{18})+\mathscr{P}^3(x_{12})x_{18}=x_{12}\mathscr{P}^3(x_{18})+x_{18}^2$ which is not equal to $\pm x_{12}^3$. Accordingly, either case is impossible to be the type of $X$.

We notice that $(r,n,m)=(2,3,6)$ is impossible directly by the above lemma.


For the remaining cases which do not appear in the final list, we can check whether the condition (\ref{condition}) fails or not in an appropriate $\{\psi^k\}$-module $K^\prime$ constructed from $K$ (with the help of computer), and find that when $(r,n,m)=(2, 3, 9)$, $(2, 21, 27)$, $(2,30,36)$, $(2,39,45)$, $(18,24,26)$, $(16,30,36)$, $(19, 30,36)$, $(21,27,29)$ or $(30,36,38)$, (\ref{condition}) holds which implies $X$ can not be ${\rm mod}~3$ $A_3$-space.

{\protect\vspace{-2pt}\rightline{$\square$}}

\section*{Acknowledgements}
The authors would like to thank Prof. Stephen D. Theriault and Prof. Mamoru Mimura for helpful discussions and comments, and are also indebted to Prof. John R. Harper for suggesting the reference \cite{Hemmi} and valuable knowledge about $H$-space of rank $p$ and higher associativity. We wish to thank the referee most warmly for his/her suggestions and comments on using the modified projective space of Hemmi \cite{Hemmi2} which has essentially improved the article, and also the careful reading of our manuscript. We are also indebted to Prof. J\'{e}r\^{o}me Scherer and Prof. Fred Cohen for careful reading of the manuscript and many valuable suggestions which have improved the paper. 

The authors are partially supported by the Singapore Ministry of Education research grant (AcRF Tier 1 WBS No. R-146-000-222-112). The second author is also supported by a grant (No. 11329101) of NSFC of China.

\end{document}